\documentclass[11pt,oneside,reqno]{amsart}
\usepackage{amssymb,mathrsfs,graphicx,enumerate}
\usepackage{amsmath,amsfonts,amssymb,amscd,amsthm,bbm}
\usepackage{colortbl}
\usepackage{cite}
\usepackage{subcaption}
\usepackage{nccmath}
\usepackage{kotex}
\usepackage{bm}
\usepackage{upgreek}
\title[Moving parallel transport in Riemannian Cucker--Smale dynamics]{Geometric Control of Moving Parallel Transport in Riemannian Cucker--Smale Dynamics with Bonding Forces}

\author[H. Ahn]{Hyunjin Ahn}
\address[Hyunjin Ahn]{\newline Department of Data Technology\newline 
	Myongji University, Seoul 03674, Republic of Korea}
\email{ahj92@mju.ac.kr}

\author[W. Shim]{Woojoo Shim$^*$}
\address[Woojoo Shim]{\newline Department of Mathematics Education \newline Kyungpook National University, Daegu 41566, Republic of Korea}
\email{wjshim@knu.ac.kr}

\subjclass[2020]{34C40, 34D05, 37N35, 53C21, 58K55, 92D50, 93A16}

\keywords{Bonding force, Cucker--Smale model, flocking dynamics, Jacobi fields, Riemannian curvature, Riemannian manifolds}

\thanks{Acknowledgment: The work of H. Ahn was supported by the National Research Foundation of Korea (NRF) grant funded by the Korea government (MSIT) (2022R1C12007321).}
 
\thanks{$^*$ Corresponding author}

\newtheorem{theorem}{Theorem}[section]
\newtheorem{lemma}{Lemma}[section]

\newtheorem{proposition}{Proposition}[section]
\newtheorem{remark}{Remark}[section]

\newtheorem{definition}{Definition}[section]

\newcommand{\vast}{\bBigg@{4}}
\newcommand{\Vast}{\bBigg@{5}}

\begin{document}
	
	\date{\today}
	
\begin{abstract}
	We study a Cucker--Smale type system with bonding forces on complete Riemannian manifolds with uniformly bounded curvature. On general manifolds, the time variation of parallel transport between moving agents produces curvature-dependent terms, so the standard energy-dissipation argument does not directly yield asymptotic velocity alignment. The bonding energy confines all pairwise distances below the injectivity radius, providing global well-posedness and time integrability of the transported velocity discrepancies. To overcome the remaining geometric obstruction, we combine the variation formula for parallel transport with a uniform endpoint estimate for Jacobi fields along moving minimizing geodesics. This yields the uniform regularity needed to convert energy dissipation into asymptotic alignment. Under an energy-dependent injectivity condition and a positive communication bound on the dynamically relevant distance range, we establish asymptotic flocking. Numerical simulations on $\mathbb S^2\times\mathbb R$ illustrate the resulting dynamics in a nonconstant-sectional-curvature setting.
\end{abstract}
	
	\maketitle
	
	\centerline{\date}
	
	%\tableofcontents

	\section{Introduction} \label{sec:1}
\setcounter{equation}{0}
Flocking is one of the most representative collective behaviors
observed in many natural and engineered systems, such as bird flocks,
fish schools, pedestrian groups, autonomous vehicles, and robotic
swarms. In flocking phenomena, individual agents adjust their motions
through local interactions and eventually form a coherent group motion.
The mathematical modeling and analysis of such emergent dynamics have
attracted considerable attention, with particular emphasis on
understanding how local interaction rules lead to global velocity
alignment and stable group formation. When agents move
on curved or geometrically constrained configuration spaces, however,
the geometry of the underlying space becomes an intrinsic part of the
interaction mechanism. This motivates the study of flocking dynamics
in a Riemannian setting.

The Cucker--Smale model introduced in \cite{C-S} is a
standard framework for describing such collective behavior in Euclidean
spaces. Its alignment interaction dissipates velocity fluctuations and,
under suitable conditions, leads to a uniformly bounded configuration
with asymptotically aligned velocities; see the survey chapter
\cite{C-H-L}.

{
Throughout the paper, let $({\mathcal M},g)$ be a complete, connected,
smooth Riemannian manifold without boundary, let $\nabla$ be its
Levi--Civita connection, and let $d(\cdot,\cdot)$ denote the geodesic
distance. We write $T_x{\mathcal M}$ and $T{\mathcal M}$ for the tangent
space at $x\in{\mathcal M}$ and the tangent bundle, respectively.
Velocities based at distinct points belong to different tangent spaces
and cannot be compared without prescribing a transport between them.
Whenever $x,y\in{\mathcal M}$ are joined by a unique length-minimizing
geodesic, we denote by
\[
P_{xy}:T_y{\mathcal M}\longrightarrow T_x{\mathcal M}
\]
the Levi--Civita parallel transport from $y$ to $x$ along this geodesic.
Thus, for $v_i\in T_{x_i}{\mathcal M}$ and
$v_j\in T_{x_j}{\mathcal M}$, the vector
$P_{x_i x_j}v_j-v_i\in T_{x_i}{\mathcal M}$ is the transported velocity
discrepancy used in the model. In particular, $P_{xy}$ is well-defined
whenever $d(x,y)<\operatorname{inj}({\mathcal M})$. Throughout the paper,
asymptotic flocking means uniform boundedness of all pairwise distances
together with decay of these transported velocity discrepancies; the
additional geometric conventions and precise definition are given in
Section~\ref{sec:2.1}. For notational simplicity, we write
$[N]:=\{1,\ldots,N\}$.
}

The Riemannian Cucker--Smale model proposed in
\cite{H-K-S} is governed, for the position--velocity pairs
$\{(x_i,v_i)\}_{i=1}^N$, by the following Cauchy problem:
\begin{equation} \label{CSM}
	\begin{cases}
		\displaystyle
		\dot{x}_i=v_i,
		\quad t>0,
		\quad i\in[N],
		\\[0.2cm]
		\displaystyle
		\nabla_{v_i}v_i
		=
		\frac{\kappa}{N}
		\sum_{j=1}^{N}
		\phi(d(x_i,x_j))
		\left(P_{x_ix_j}v_j-v_i\right),
		\\[0.2cm]
		\displaystyle
		(x_i(0),v_i(0))
		=
		(x_i^0,v_i^0)
		\in T\mathcal M,
	\end{cases}
\end{equation}
where $N$ denotes the number of agents, $\kappa>0$ denotes the
coupling strength, and
$\phi:[0,\infty)\to[0,\infty)$ is a locally Lipschitz continuous
communication weight.

The system \eqref{CSM} and its variants have been
analyzed mainly in special geometric settings. Asymptotic flocking on
the two-dimensional sphere and hyperboloid was studied in
\cite{H-K-S}; related thermodynamic and relativistic models on spheres
and hyperboloids were considered in \cite{A-H-S,A-H-P-S,A-H-K-S}; the
infinite cylindrical domain was treated in \cite{B-H-K-K-S}; and a
mean-field limit was investigated in \cite{A-H-K-S-S}. These settings
provide compactness, explicit representations, or geometric identities
that are not available on a general manifold with variable curvature.

Two coupled difficulties prevent the Euclidean
energy-dissipation argument from carrying over directly. First, the
parallel transport $P_{x_i(t)x_j(t)}$ used to compare two velocities
varies because both endpoints move. Differentiating
\[
t\mapsto
\|P_{x_i(t)x_j(t)}v_j(t)-v_i(t)\|_{x_i(t)}^2
\]
produces curvature-dependent terms represented by
Jacobi fields along the moving minimizing geodesics. Second, the
logarithm maps and parallel transports remain smooth only while the
agents stay inside a common injectivity region, and the same distance
control is required for uniform estimates of the curvature terms. The
alignment dissipation in \eqref{CSM} does not in general provide this
spatial bound. Thus one must obtain both confinement and a geometric
estimate for the time-dependent transport.

{
Bonding interactions provide a natural mechanism for the confinement
part of the problem. Whenever
$d(x,y)<\operatorname{inj}({\mathcal M})$, we write
$\log_x y\in T_x{\mathcal M}$ for the initial velocity of the unique
length-minimizing geodesic from $x$ to $y$, parametrized on $[0,1]$.
We therefore consider the Riemannian Cucker--Smale dynamics with the
bonding term
\[
\frac{\kappa_b}{N}
\sum_{j=1}^N
\phi_b(d(x_i,x_j)^2)\log_{x_i}x_j.
\]
}
The associated potential energy controls the pairwise
distances and keeps the configuration below the injectivity radius.
This does not by itself solve the alignment problem: even after
confinement, the variation of moving parallel transport still generates
curvature-dependent terms that must be controlled without explicit
formulas or constant-curvature cancellations.

More precisely, for the position--velocity pairs
$\{(x_i,v_i)\}_{i=1}^N$, we consider the following Cauchy problem:
\begin{equation}
	\begin{cases} \label{Main}
		\vspace{0.1cm}
		\displaystyle
		\dot{x}_i=v_i,
		\quad t>0,
		\quad i\in[N],
		\\
		\vspace{0.1cm}
		\displaystyle
		\nabla_{v_i}v_i
		=
		\frac{\kappa_f}{N}
		\sum_{j=1}^{N}
		\phi_f(d(x_i,x_j))
		\left(P_{x_ix_j}v_j-v_i\right)
		+
		\frac{\kappa_b}{N}
		\sum_{j=1}^N
		\phi_b(d(x_i,x_j)^2)
		\log_{x_i}x_j,
		\\
		\displaystyle
		(x_i(0),v_i(0))
		=
		(x_i^0,v_i^0)
		\in T\mathcal M,
	\end{cases}
\end{equation}
where $N$ denotes the number of agents, $\kappa_f>0$ and
$\kappa_b>0$ denote the coupling strengths for the flocking and
bonding interactions, respectively. Moreover,
$\phi_f:[0,\infty)\to[0,\infty)$ and
$\phi_b:[0,\infty)\to[0,\infty)$ are locally Lipschitz continuous
communication and bonding kernels, respectively.

Cucker--Smale type systems with bonding forces have
previously been studied in Euclidean spaces and on the circle
\cite{P-K-H,A-B-H-Y1}, in relativistic and Riemannian settings related
to global well-posedness and collision avoidance \cite{A-B-H-Y2}, on
spheres and hyperboloids \cite{A-B-H-Y}, and on manifolds with constant
sectional curvature \cite{Q-L-W}. These works show that bonding forces
are effective for spatial confinement and, in special geometries, for
proving flocking. Under only a uniform curvature bound, however,
differentiating the transported discrepancy
\[
\|P_{x_i x_j}v_j-v_i\|_{x_i}
\]
produces a curvature integral involving a Jacobi field
along a minimizing geodesic whose two endpoints move with the agents.
Consequently, confinement alone does not imply
\[
\lim_{t\to\infty}
\|P_{x_i(t)x_j(t)}v_j(t)-v_i(t)\|_{x_i(t)}
=
0.
\]
A uniform endpoint-to-interior estimate for this family
of moving-endpoint Jacobi fields is therefore needed. This is the
analytical point at which the bounded-variable-curvature setting differs
from the special geometries treated previously.

\vspace{0.2cm}

{
	The preceding discussion leads to the following main
	issue:
	\begin{itemize}
		\item
		\emph{(Main issue):}
		Can the curvature contribution generated by moving parallel
		transport be controlled uniformly under an energy-dependent
		injectivity-radius confinement and a uniform curvature bound, so
		that energy dissipation implies asymptotic alignment without
		explicit formulas or constant-curvature cancellations?
	\end{itemize}
}

In this work, we answer this question affirmatively.
The energy dissipation identity keeps all pairwise distances uniformly
below the injectivity radius, provides a uniform velocity bound and
global well-posedness, and yields time integrability of the squared
transported velocity discrepancies. To convert this integrability into
pointwise decay, we formulate the variation of parallel transport along
the moving minimizing geodesics and isolate a curvature-dependent
remainder. We then establish a uniform endpoint estimate for the
associated Jacobi fields over all geodesics whose lengths remain below
the bonding-energy radius. Combined with the curvature bound, this
estimate gives a uniform bound for the time derivative of the squared
transported velocity discrepancies. Barbalat's lemma yields asymptotic
velocity alignment, and the confinement estimate completes the proof of
asymptotic flocking. The result holds under an energy-dependent
injectivity condition, a positive lower bound on the communication
kernel over the dynamically relevant distance range, and a uniform
bound on the Riemannian curvature tensor.

The rest of this paper is organized as follows. In
Section~\ref{sec:2}, we fix the geometric notation, state the precise
definition of asymptotic flocking, and recall Barbalat's lemma together
with the required facts on curvature and Jacobi fields. In
Section~\ref{sec:3}, we establish global well-posedness and asymptotic
flocking for \eqref{Main}. Section~\ref{sec:4} presents numerical
simulations on $\mathbb S^2\times\mathbb R$, and
Section~\ref{sec:5} concludes the paper.

\section{Preliminaries} \label{sec:2}
\setcounter{equation}{0}
In this section, we fix the geometric notation and the
notion of asymptotic flocking used throughout the paper. We then recall
Barbalat's lemma and the basic facts on the Riemannian curvature tensor,
sectional curvature, and Jacobi fields needed in the subsequent
analysis.

\subsection{Geometric identities and asymptotic flocking} \label{sec:2.1}

{
We retain the notation introduced in Section~\ref{sec:1} and record here
only the additional constructions and identities used in the analysis.
For $v\in T_x{\mathcal M}$, we write
$\|v\|_x:=\sqrt{g_x(v,v)}$.

Let $x,y\in{\mathcal M}$ be joined by a unique length-minimizing
geodesic $\eta_{xy}:[0,1]\to{\mathcal M}$, parametrized so that
$\eta_{xy}(0)=y$ and $\eta_{xy}(1)=x$. For
$w\in T_y{\mathcal M}$, let $W$ be the vector field along
$\eta_{xy}$ satisfying
$\nabla_{\dot\eta_{xy}(s)}W(s)=0$ and $W(0)=w$. Consistently with the
notation introduced in Section~\ref{sec:1}, we then have
$P_{xy}w=W(1)$. Since parallel transport is an isometry, for
$w,w'\in T_y{\mathcal M}$,
\begin{equation} \label{new1}
    P_{xy}P_{yx}=\mathrm{Id}_{T_x{\mathcal M}},\quad
    \|P_{xy}w\|_x=\|w\|_y,\quad
    g_x(P_{xy}w,P_{xy}w')=g_y(w,w'),
\end{equation}
where $\mathrm{Id}_{T_x{\mathcal M}}$ denotes the identity map on
$T_x{\mathcal M}$.

We write $\exp_x$ for the exponential map at $x$. Whenever $y$ lies
in the injectivity domain of $x$, the logarithm map is given by
$\log_x y:=\exp_x^{-1}(y)\in T_x{\mathcal M}$. We denote by
$\operatorname{Cut}(x)$ the cut locus of $x$ and by
\[
\operatorname{inj}({\mathcal M})
:=
\inf_{x\in{\mathcal M}}\operatorname{inj}_x({\mathcal M})
\]
the injectivity radius of ${\mathcal M}$. Thus,
$d(x,y)<\operatorname{inj}({\mathcal M})$ guarantees that the unique
length-minimizing geodesic, $P_{xy}$, and $\log_x y$ used above are
well-defined.

We next recall the identities from \cite{A-B-H-Y2} that will be used
repeatedly. Let $\tau\in(0,\infty]$, and let $x_i(\cdot)$ and
$x_j(\cdot)$ be $C^1$ curves on ${\mathcal M}$ defined on
$[0,\tau)$ with
\[
\dot x_i(t)=v_i(t)\in T_{x_i(t)}{\mathcal M}
\quad\text{and}\quad
\dot x_j(t)=v_j(t)\in T_{x_j(t)}{\mathcal M}.
\]
If
\[
\sup_{0\le t<\tau}d(x_i(t),x_j(t))
<
\operatorname{inj}({\mathcal M}),
\]
then, for all $t\in[0,\tau)$,
\begin{align}\label{new2}
    \begin{aligned}
        & P_{x_i x_j}\big(\log_{x_j}x_i\big)
        =
        -\log_{x_i}x_j,
        \quad
        \|\log_{x_i}x_j\|_{x_i}
        =
        d(x_i,x_j),
        \\
        & \frac{d}{dt}d(x_i,x_j)^2
        =
        2g_{x_i}\big(
        P_{x_i x_j}v_j-v_i,\log_{x_i}x_j
        \big).
    \end{aligned}
\end{align}
For standard background on Riemannian geometry, we refer the reader
to \cite{Ju,P}.

We now state the precise notion of flocking used in this paper.
\begin{definition} \label{D1.1}~
    \emph{(Asymptotic flocking)}
    Let
    \[
    Z(t)=\{(x_i(t),v_i(t))\}_{i=1}^N
    \]
    be an $N$-agent trajectory in $(T\mathcal M)^N$. Assume that the
    parallel transport $P_{x_i(t)x_j(t)}$ is well-defined for all
    $i,j\in [N]$ and $t\ge0$. We say that $Z$ exhibits asymptotic
    flocking if
    \[
    \sup_{t\ge0}\max_{i,j\in[N]}d(x_i(t),x_j(t))<\infty
    \]
    and
    \[
    \lim_{t\to\infty}
    \max_{i,j\in[N]}
    \|P_{x_i(t)x_j(t)}v_j(t)-v_i(t)\|_{x_i(t)}
    =0.
    \]
\end{definition}
}

\subsection{Barbalat's lemma} \label{sec:2.2}
In this subsection, we recall a standard form of Barbalat's lemma. This lemma will be used later to prove the velocity alignment of the system \eqref{Main}.

\begin{lemma} \label{barbalat} \emph{\cite{B}}~\emph{(Barbalat's lemma)}
	Let $f:[0,\infty)\to \mathbb{R}$ be uniformly continuous. Assume that
	\begin{equation*}
		\int_0^\infty f(s)\,ds
	\end{equation*}
	exists and is finite. Then, it follows that
	\begin{equation*}
		\lim_{t\to\infty} f(t)=0.
	\end{equation*}
\end{lemma}

Barbalat's lemma is useful because it converts a global-in-time integrability estimate into pointwise convergence. In our analysis, once the transported velocity discrepancy
\[
\|P_{x_i(t)x_j(t)}v_j(t)-v_i(t)\|_{x_i(t)}^2
\]
is shown to be integrable in time and uniformly continuous, this lemma yields
\[
\lim_{t\to\infty}
\|P_{x_i(t)x_j(t)}v_j(t)-v_i(t)\|_{x_i(t)}
=0.
\]

\subsection{Riemannian curvature tensor, sectional curvature, and Jacobi fields} \label{sec:2.3}
In this subsection, we recall some basic differential-geometric notions related to the Riemannian curvature tensor, sectional curvature, and Jacobi fields. These notions will be used later to control the time variation of parallel transport and to prove the uniform boundedness of the time derivative of the transported velocity discrepancy
\[
\|P_{x_i(t)x_j(t)}v_j(t)-v_i(t)\|_{x_i(t)}^2.
\]

\begin{definition} \label{D2.1} \emph{\cite{Ju,P}}~\emph{(Riemannian curvature, sectional curvature, and Jacobi fields)}
	Let $(\mathcal M,g)$ be a Riemannian manifold and let $\nabla$ be the Levi-Civita connection on $\mathcal M$.
	\begin{enumerate}
		\item \emph{(Riemannian curvature tensor):} The Riemannian curvature tensor $R$ is defined by
		\begin{equation*}
			R(X,Y)Z
			:=
			\nabla_X\nabla_Y Z
			-
			\nabla_Y\nabla_X Z
			-
			\nabla_{[X,Y]}Z
		\end{equation*}
		for smooth vector fields $X,Y,Z$ on $\mathcal M$.
		{
			For $x\in\mathcal M$, we define the operator norm of the curvature tensor by
			\[
			\|R_x\|_{\mathrm{op}}
			:=
			\sup_{\substack{u,v,w\in T_x\mathcal M\\
					\|u\|_x=\|v\|_x=\|w\|_x=1}}
			\|R_x(u,v)w\|_x.
			\]
		}
		\vspace{0.2cm}
		
		\item \emph{(Sectional curvature):} For $x\in\mathcal M$ and linearly independent tangent vectors $u,w\in T_x\mathcal M$, the sectional curvature of the two-dimensional subspace
		\[
		\sigma:=\operatorname{span}\{u,w\}\subset T_x\mathcal M
		\]
		is defined by
		\begin{equation*}
			\sec_x(\sigma)
			:=
			\frac{
				g_x\big(R(u,w)w,u\big)
			}{
				g_x(u,u)g_x(w,w)-g_x(u,w)^2
			}.
		\end{equation*}
		This quantity is independent of the choice of linearly independent vectors $u,w$ spanning $\sigma$.
		\vspace{0.2cm}
		
		\item \emph{(Jacobi fields):} Let $\gamma:[0,1]\to \mathcal M$ be a geodesic, where $s\in[0,1]$ denotes the parameter along the geodesic. A smooth vector field $J$ along $\gamma$ is called a Jacobi field if it satisfies the Jacobi equation
		\begin{equation*}
			\nabla_s^2 J
			+
			R(J,\dot{\gamma})\dot{\gamma}
			=
			0,
			\quad s\in[0,1],
		\end{equation*}
		where $\dot{\gamma}:=d\gamma/ds$, $\nabla_s:=\nabla_{\dot{\gamma}}$ denotes the covariant derivative along $\gamma$, and $\nabla_s^2J:=\nabla_s(\nabla_sJ)$.
	\end{enumerate}
\end{definition}

The Riemannian curvature tensor measures how covariant derivatives fail to commute, and it also describes how parallel transport changes when the underlying path varies. The sectional curvature measures the curvature of a two-dimensional tangent plane and provides a scalar way to formulate curvature bounds on the underlying manifold. A Jacobi field describes the infinitesimal variation of a family of geodesics. In our analysis, the Riemannian curvature tensor and Jacobi fields are used to estimate the curvature-dependent terms that appear when differentiating
$P_{x_i(t)x_j(t)}v_j(t)$ with respect to time, while sectional curvature is recalled for completeness and used to describe the geometric backgrounds considered in the numerical simulations in Section \ref{sec:4}.

\section{Asymptotic flocking} \label{sec:3}
	\setcounter{equation}{0}
	In this section, we prove the global well-posedness and asymptotic flocking of the system \eqref{Main} on a broad class of Riemannian manifolds under the assumption of bounded Riemannian curvature. In particular, our analysis is not restricted to special constant-curvature spaces.
	
	\subsection{Sufficient conditions} \label{sec:3.1}
In this subsection, we present sufficient conditions for asymptotic flocking in terms of the initial data, system parameters, communication and bonding kernels, and the geometry of the underlying Riemannian manifold. We also clarify the meaning of these conditions.

\vspace{0.2cm}

\begin{itemize}
	\item
	$(\mathcal F_1)$ {(Kernels and system parameters):}
	The communication kernel $\phi_f:[0,\infty)\to[0,\infty)$ and the bonding kernel $\phi_b:[0,\infty)\to[0,\infty)$ are locally Lipschitz continuous, and the coupling strengths satisfy
	\[
	\kappa_f>0,\quad \kappa_b>0.
	\]
	
	\item
	$(\mathcal F_2)$ {(Initial energy and injectivity condition):}
We define the bonding-energy radius $U$ and assume that it satisfies the following injectivity condition:
\begin{equation*}
	0<U
	:=
	\sup\left\{
	r\ge0~\bigg|~
	\frac{\kappa_b}{2N}
	\int_0^{r^2}\phi_b(s)\,ds
	\le
	\mathcal E(0)
	\right\}<\operatorname{inj}(\mathcal M),
\end{equation*}
	where the initial energy $\mathcal E(0)$ is defined by
	\begin{equation*}
		\mathcal E(0)
		:=
		\frac12\sum_{i=1}^N\|v_i^0\|_{x_i^0}^2
		+
		\frac{\kappa_b}{4N}
		\sum_{\substack{i,j=1}}^N
		\int_0^{d(x_i^0,x_j^0)^2}\phi_b(s)\,ds.
	\end{equation*}

\item
$(\mathcal F_3)$ {(Positive communication on bounded distances):}
We assume that the communication kernel has a positive lower bound on the relevant distance range, namely
\begin{equation*}
	\underline{\phi_f}
	:=
	\inf_{0\le r\le U}\phi_f(r)
	>
	0.
\end{equation*}
	
	\item
	$(\mathcal F_4)$ {(Bounded Riemannian curvature):}
	The Riemannian curvature tensor of $(\mathcal M,g)$ is uniformly bounded, i.e., there exists a constant $C_R>0$ such that
	\begin{equation*}
		{
			\|R_x\|_{\mathrm{op}}\le C_R,
			\quad x\in\mathcal M.
		}
	\end{equation*}
\end{itemize}
	
\begin{remark} \label{R3.1}~\emph{(Role of the sufficient conditions)}
	The assumptions $(\mathcal F_1)-(\mathcal F_4)$ have the following roles in our analysis.
	\begin{enumerate}
	\item
	The assumption $(\mathcal F_1)$ provides the basic regularity and sign conditions for the model. {
		The local Lipschitz continuity of the communication and bonding
		kernels, together with the smooth dependence of the logarithm map
		and parallel transport on their endpoints inside the injectivity
		region, is used to obtain the local well-posedness of \eqref{Main}.
	} The nonnegativity of these kernels is essential for constructing a nonnegative energy functional and deriving its dissipation structure. More precisely, the nonnegativity of $\phi_b$ makes the bonding energy nonnegative, while the nonnegativity of $\phi_f$ yields the dissipative alignment term in the energy estimate. The positivity of $\kappa_f$ and $\kappa_b$ ensures that these two mechanisms act in the desired direction.
		\vspace{0.2cm}
{
	\item
	The assumption $(\mathcal F_2)$ links the initial energy with the
	bonding interaction. The radius $U$ is determined by the bonding
	potential energy, and the condition
	\[
	U<\operatorname{inj}(\mathcal M)
	\]
	will be used to prove
	\[
	d(x_i(t),x_j(t))
	\leq U
	<
	\operatorname{inj}(\mathcal M),
	\quad
	t\geq0,\quad i,j\in[N].
	\]
	Since all pairwise distances remain uniformly below the injectivity
	radius, any two agents are connected by a unique length-minimizing
	geodesic. Hence, the logarithm maps and parallel transports appearing
	in \eqref{Main} remain well-defined along the solution. Moreover, this
	uniform distance bound gives the first requirement in the definition
	of asymptotic flocking in Definition~\ref{D1.1}, namely the formation
	of a spatially bounded group. Together with the velocity bound obtained
	from the energy estimate, this distance control provides the a priori
	bounds used to establish the global well-posedness in
	Proposition~\ref{P3.1}.
}
			\vspace{0.2cm}
		\item
		The assumption $(\mathcal F_3)$ is essential for proving the asymptotic decay of the transported velocity discrepancies. Once the inter-agent distances are bounded by $U$, this condition gives a positive lower bound for the communication kernel on the relevant distance range. Hence, the energy dissipation estimate yields the time integrability of
		$
		\|P_{x_i(t)x_j(t)}v_j(t)-v_i(t)\|_{x_i(t)}^2
		$
		on $[0,\infty)$. This integrability estimate is the first key step toward velocity alignment via Barbalat's lemma (Lemma \ref{barbalat}), {
			once a uniform derivative bound for the squared transported velocity
			discrepancy is obtained.
		}
			\vspace{0.2cm}
		\item
The assumption $(\mathcal F_4)$ is needed to control the geometric terms arising from the time variation of parallel transport. More precisely, using Jacobi-field estimates together with the boundedness of the Riemannian curvature tensor, we will prove that the time derivative of
$
\|P_{x_i(t)x_j(t)}v_j(t)-v_i(t)\|_{x_i(t)}^2
$
is uniformly bounded. Once the time integrability of this quantity is obtained, this uniform derivative bound will allow us to apply Barbalat's lemma (Lemma \ref{barbalat}) and conclude the asymptotic velocity alignment:
\[
\lim_{t\to\infty}
\max_{i,j\in[N]}
\|P_{x_i(t)x_j(t)}v_j(t)-v_i(t)\|_{x_i(t)}
=0.
\]
	\end{enumerate}
\end{remark}

\subsection{Asymptotic flocking} \label{sec:3.2}
In this subsection, we study the asymptotic flocking of the system \eqref{Main} under the sufficient conditions $(\mathcal F_1)-(\mathcal F_4)$. We first introduce the kinetic energy, the bonding energy, and the total energy as follows: {
	Let $\{(x_i,v_i)\}_{i=1}^N$ be a solution to \eqref{Main}
	defined on an interval $[0,T)$. For $0\leq t<T$, we define
}
\begin{align*}
	&\mathcal E_k(t)
	:=
	\frac12\sum_{i=1}^N\|v_i(t)\|_{x_i(t)}^2,\quad \mathcal E_b(t)
	:=
	\frac{\kappa_b}{4N}
	\sum_{\substack{i,j=1\\ i\neq j}}^N
	\int_0^{d(x_i(t),x_j(t))^2}\phi_b(s)\,ds,\\
	&\mathcal E(t)
	:=
	\mathcal E_k(t)+\mathcal E_b(t).
\end{align*}

We first show that the total energy $\mathcal E$ is nonincreasing along solutions.
\begin{lemma} \label{L3.1}~
	\emph{(Energy functional)}
	Let $\{(x_i,v_i)\}_{i=1}^N$ be a solution to \eqref{Main}
	defined on an interval $[0,T)$ and satisfying
	\[
	d(x_i(t),x_j(t))
	<
	\operatorname{inj}(\mathcal M),
	\quad
	0\leq t<T,\quad i,j\in[N].
	\]
	Assume $(\mathcal F_1)$. Then, for $0<t<T$, the total energy
	satisfies
	\[
	\frac{d}{dt}\mathcal E(t)
	=
	-\frac{\kappa_f}{2N}
	\sum_{i,j=1}^N
	\phi_f(d(x_i(t),x_j(t)))
	\|P_{x_i(t)x_j(t)}v_j(t)-v_i(t)\|_{x_i(t)}^2
	\leq0.
	\]
	In particular,
	\[
	\mathcal E(t)\leq\mathcal E(0),
	\quad 0\leq t<T.
	\]
\end{lemma}

\begin{proof}
	By differentiating the kinetic energy and using \eqref{Main}, we find
	\begin{align*}
		\frac{d}{dt}\mathcal E_k(t)
		&=
		\sum_{i=1}^N
		g_{x_i(t)}\big(\nabla_{v_i(t)}v_i(t),v_i(t)\big)\\
		&=
		\frac{\kappa_f}{N}
		\sum_{i,j=1}^N
		\phi_f(d(x_i(t),x_j(t)))
		g_{x_i(t)}
		\big(
		P_{x_i(t)x_j(t)}v_j(t)-v_i(t),v_i(t)
		\big)\\
		&\quad
		+
		\frac{\kappa_b}{N}
		\sum_{i,j=1}^N
		\phi_b(d(x_i(t),x_j(t))^2)
		g_{x_i(t)}
		\big(
		\log_{x_i(t)}x_j(t),v_i(t)
		\big).
	\end{align*}
Using the standard symmetrization trick of interchanging $i$ and $j$, together with the isometry property of parallel transport in \eqref{new1}, the flocking interaction gives
\begin{align*}
	&\frac{\kappa_f}{N}
	\sum_{i,j=1}^N
	\phi_f(d(x_i(t),x_j(t)))
	g_{x_i(t)}
	\big(
	P_{x_i(t)x_j(t)}v_j(t)-v_i(t),v_i(t)
	\big)\\
	&=
	\frac{\kappa_f}{2N}
	\sum_{i,j=1}^N
	\phi_f(d(x_i(t),x_j(t)))
	\Big[
	g_{x_i(t)}
	\big(
	P_{x_i(t)x_j(t)}v_j(t)-v_i(t),v_i(t)
	\big)\\
	&\hspace{5.5cm}
	+
	g_{x_j(t)}
	\big(
	P_{x_j(t)x_i(t)}v_i(t)-v_j(t),v_j(t)
	\big)
	\Big]\\
	&=
	\frac{\kappa_f}{2N}
	\sum_{i,j=1}^N
	\phi_f(d(x_i(t),x_j(t)))
	\Big[
	g_{x_i(t)}
	\big(
	P_{x_i(t)x_j(t)}v_j(t)-v_i(t),v_i(t)
	\big)\\
	&\hspace{5.5cm}
	-
	g_{x_i(t)}
	\big(
	P_{x_i(t)x_j(t)}v_j(t)-v_i(t),
	P_{x_i(t)x_j(t)}v_j(t)
	\big)
	\Big]\\
	&=
	-\frac{\kappa_f}{2N}
	\sum_{i,j=1}^N
	\phi_f(d(x_i(t),x_j(t)))
	\|P_{x_i(t)x_j(t)}v_j(t)-v_i(t)\|_{x_i(t)}^2.
\end{align*}
	On the other hand, by the distance identity in \eqref{new2}, we have
	\begin{align*}
		\frac{d}{dt}\mathcal E_b(t)
		&=
		\frac{\kappa_b}{4N}
		\sum_{\substack{i,j=1\\ i\neq j}}^N
		\phi_b(d(x_i(t),x_j(t))^2)
		\frac{d}{dt}d(x_i(t),x_j(t))^2\\
		&=
		\frac{\kappa_b}{2N}
		\sum_{\substack{i,j=1\\ i\neq j}}^N
		\phi_b(d(x_i(t),x_j(t))^2)
		g_{x_i(t)}
		\big(
		P_{x_i(t)x_j(t)}v_j(t)-v_i(t),
		\log_{x_i(t)}x_j(t)
		\big)\\
		&=
		-\frac{\kappa_b}{N}
		\sum_{\substack{i,j=1\\ i\neq j}}^N
		\phi_b(d(x_i(t),x_j(t))^2)
		g_{x_i(t)}
		\big(
		\log_{x_i(t)}x_j(t),v_i(t)
		\big).
	\end{align*}
	Indeed, by interchanging $i$ and $j$ and using \eqref{new1},
	\[
	\begin{aligned}
		&\sum_{\substack{i,j=1\\i\neq j}}^N
		\phi_b(d(x_i,x_j)^2)
		g_{x_i}
		\left(
		P_{x_ix_j}v_j,\log_{x_i}x_j
		\right)
		\\
		&\quad=
		\sum_{\substack{i,j=1\\i\neq j}}^N
		\phi_b(d(x_i,x_j)^2)
		g_{x_i}
		\left(
		v_i,P_{x_ix_j}\log_{x_j}x_i
		\right)
		\\
		&\quad=
		-
		\sum_{\substack{i,j=1\\i\neq j}}^N
		\phi_b(d(x_i,x_j)^2)
		g_{x_i}
		\left(
		v_i,\log_{x_i}x_j
		\right).
	\end{aligned}
	\]
	Therefore, the bonding contribution in $\frac{d}{dt}\mathcal E_k(t)$ is exactly canceled by $\frac{d}{dt}\mathcal E_b(t)$. Combining the above identities, we derive the desired assertion.
\end{proof}

	In what follows, using Lemma \ref{L3.1}, we obtain uniform bounds for the velocities and the inter-agent distances.
	
{
	\begin{proposition} \label{P3.1}~
		\emph{(Global well-posedness and energy estimates)}
		Assume $(\mathcal F_1)-(\mathcal F_2)$. Then, the Cauchy problem
		\eqref{Main} admits a unique global solution $\{(x_i,v_i)\}_{i=1}^N$. Moreover, for all $t\geq0$,
		\[
		\max_{i\in[N]}
		\|v_i(t)\|_{x_i(t)}
		\leq
		\sqrt{2\mathcal E(0)}
		\]
		and
		\[
		\max_{i,j\in[N]}
		d(x_i(t),x_j(t))
		\leq
		U
		<
		\operatorname{inj}(\mathcal M).
		\]
	\end{proposition}
	
	\begin{proof}
		We first verify that the initial configuration lies inside the
		injectivity region. For each $i,j\in[N]$, the symmetry of the distance
		and the definition of $\mathcal E(0)$ give
		\[
		\frac{\kappa_b}{2N}
		\int_0^{d(x_i^0,x_j^0)^2}
		\phi_b(s)\,ds
		\leq
		\mathcal E(0).
		\]
		Hence, by the definition of $U$ in $(\mathcal F_2)$,
		\[
		d(x_i^0,x_j^0)
		\leq
		U
		<
		\operatorname{inj}(\mathcal M).
		\]
		Consider the open subset
		\[
		\mathscr D
		:=
		\left\{
		\bigl((x_i,v_i)\bigr)_{i=1}^N\in(T\mathcal M)^N
		\;\middle|\;
		d(x_i,x_j)<\operatorname{inj}(\mathcal M)
		\text{ for all }i,j\in[N]
		\right\}.
		\]
On $\mathscr D$, the maps
\[
(x,y)\mapsto\log_x y
\quad\text{and}\quad
(x,y,w)\mapsto P_{xy}w
\]
are smooth, while the distance function is locally Lipschitz.
Together with the local Lipschitz continuity of $\phi_f$ and
$\phi_b$, this shows that the geodesic spray and the force terms
in \eqref{Main} define a locally Lipschitz vector field on
$\mathscr D$. Hence, the standard local existence and uniqueness
theorem for ordinary differential equations on manifolds yields
a unique solution on a maximal interval $[0,T_{\max})$. Lemma~\ref{L3.1}, applied on this maximal interval, gives
		\[
		\mathcal E(t)\leq\mathcal E(0),
		\quad 0\leq t<T_{\max}.
		\]
		Since
		\[
		\mathcal E_k(t)
		=
		\frac12
		\sum_{k=1}^N
		\|v_k(t)\|_{x_k(t)}^2
		\leq
		\mathcal E(t),
		\]
		we obtain, for each $i\in[N]$,
		\[
		\begin{aligned}
			\|v_i(t)\|_{x_i(t)}
			&\leq
			\left(
			\sum_{k=1}^N
			\|v_k(t)\|_{x_k(t)}^2
			\right)^{1/2}
			\\
			&=
			\sqrt{2\mathcal E_k(t)}
			\leq
			\sqrt{2\mathcal E(0)},
			\quad
			0\leq t<T_{\max}.
		\end{aligned}
		\]
		We next estimate the inter-agent distances. Since $\phi_b$ is
		nonnegative, for each $i,j\in[N]$,
		\[
		\begin{aligned}
			\frac{\kappa_b}{2N}
			\int_0^{d(x_i(t),x_j(t))^2}
			\phi_b(s)\,ds
			&\leq
			\frac{\kappa_b}{4N}
			\sum_{k,\ell=1}^N
			\int_0^{d(x_k(t),x_\ell(t))^2}
			\phi_b(s)\,ds
			\\
			&=
			\mathcal E_b(t)
			\leq
			\mathcal E(t)
			\leq
			\mathcal E(0).
		\end{aligned}
		\]
		By the definition of $U$ in $(\mathcal F_2)$, this implies
		\[
		d(x_i(t),x_j(t))
		\leq U,
		\quad 0\leq t<T_{\max}.
		\]
		Consequently,
		\[
		\max_{i,j\in[N]}
		d(x_i(t),x_j(t))
		\leq
		U
		<
		\operatorname{inj}(\mathcal M),
		\quad 0\leq t<T_{\max}.
		\]
		It remains to prove that $T_{\max}=\infty$. Suppose to the contrary
		that $T_{\max}<\infty$, and set
		\[
		V:=\sqrt{2\mathcal E(0)}.
		\]
		For every $i\in[N]$ and $0\leq t<T_{\max}$, the velocity bound gives
		\[
		d(x_i(t),x_i^0)
		\leq
		\int_0^t
		\|v_i(s)\|_{x_i(s)}\,ds
		\leq
		VT_{\max}.
		\]
		Since $\mathcal M$ is complete, the Hopf--Rinow theorem implies that
		the closed metric balls
		\[
		K_i
		:=
		\overline{B}_{\mathcal M}
		\bigl(x_i^0,VT_{\max}\bigr)
		\]
		are compact. The solution therefore remains in
		\[
		\begin{aligned}
			\mathscr K
			:=
			\Bigl\{
			\bigl((x_i,v_i)\bigr)_{i=1}^N\in(T\mathcal M)^N
			\ \Big|\ 
			&x_i\in K_i,\quad
			\|v_i\|_{x_i}\leq V,
			\\
			&d(x_i,x_j)\leq U
			\quad\text{for all }i,j\in[N]
			\Bigr\}.
		\end{aligned}
		\]
		The closed disk bundle of radius $V$ over each compact set $K_i$ is
		compact. Hence, $\mathscr K$ is a compact subset of
		$(T\mathcal M)^N$. Moreover, $U<\operatorname{inj}(\mathcal M)$
		implies that $
		\mathscr K\subset\mathscr D$.
		Thus, the solution remains in a compact subset of the domain on which
		the vector field associated with \eqref{Main} is locally Lipschitz.
		The standard continuation theorem therefore extends the solution
		beyond $T_{\max}$, contradicting the maximality of $T_{\max}$. Hence, we have $
		T_{\max}=\infty$. This proves the global existence and uniqueness of the solution,
		together with the stated energy estimates.
	\end{proof}
}
  
{
	Having established global well-posedness, we next use
	$(\mathcal F_3)$ and the energy dissipation identity to obtain the
	time integrability of the transported velocity discrepancies.
}

\begin{proposition} \label{P3.2}~
	\emph{(Time integrability of transported velocity discrepancies)}
	Assume $(\mathcal F_1)-(\mathcal F_3)$, and let
	$\{(x_i,v_i)\}_{i=1}^N$ be the unique global solution to
	\eqref{Main} given by Proposition~\ref{P3.1}. Then, we have
	\begin{equation*}
		\int_0^\infty
		\sum_{i,j=1}^N
		\|P_{x_i(t)x_j(t)}v_j(t)-v_i(t)\|_{x_i(t)}^2\,dt
		\le
		\frac{2N}{\kappa_f\underline{\phi_f}}\mathcal E(0)<\infty,
	\end{equation*}
	where $\underline{\phi_f}$ is defined in $(\mathcal F_3)$.
\end{proposition}

\begin{proof}
Integrating the energy dissipation identity in Lemma \ref{L3.1} over $[0,t]$, we obtain
\begin{align*}
	\frac{\kappa_f}{2N}
	\int_0^t
	\sum_{i,j=1}^N
	\phi_f(d(x_i(s),x_j(s)))
	\|P_{x_i(s)x_j(s)}v_j(s)-v_i(s)\|_{x_i(s)}^2\,ds
	=
	\mathcal E(0)-\mathcal E(t)
	\le
	\mathcal E(0).
\end{align*}
	On the other hand, we combine the second assertion of Proposition \ref{P3.1} and $(\mathcal F_3)$ to observe
	\[
	\phi_f(d(x_i(t),x_j(t)))
	\ge
	\underline{\phi_f}=	\inf_{0\le r\le U}\phi_f(r)
	>
	0,
	\quad t\ge0,\quad i,j\in[N].
	\]
	Therefore, combining the above relations leads to
	\begin{align*}
		\int_0^t
		\sum_{i,j=1}^N
		\|P_{x_i(s)x_j(s)}v_j(s)-v_i(s)\|_{x_i(s)}^2\,ds
		\le
			\frac{2N}{\kappa_f\underline{\phi_f}}\mathcal E(0),
	\end{align*}
which implies that 
	\begin{equation*}
	\int_0^\infty
	\sum_{i,j=1}^N
	\|P_{x_i(t)x_j(t)}v_j(t)-v_i(t)\|_{x_i(t)}^2\,dt
	\le
	\frac{2N}{\kappa_f\underline{\phi_f}}\mathcal E(0)<\infty.
\end{equation*}
This completes the proof.
\end{proof}
	
In the following lemma, we provide a variation formula for parallel transport. This formula shows that, when differentiating the squared transported velocity discrepancy, additional terms appear from the time variation of the connecting geodesic. These terms are expressed through the Riemannian curvature tensor and the Jacobi field generated by the moving endpoints.

\begin{lemma} \label{L3.2}~
	\emph{(Variation formula for transported velocity discrepancies)}
	Let
	\[
	x_i,x_j:[0,T)\to\mathcal M
	\]
	be $C^2$ curves, and set
	\[
	v_i(t):=\dot{x}_i(t),
	\quad
	v_j(t):=\dot{x}_j(t).
	\]
	Assume that
	\[
	d(x_i(t),x_j(t))
	<
	\operatorname{inj}(\mathcal M),
	\quad
	0\leq t<T.
	\]
	For $0\leq s\leq1$, define
	\[
	\gamma_{ij}(t,s)
	:=
	\exp_{x_j(t)}
	\left(
	s\log_{x_j(t)}x_i(t)
	\right).
	\]
Then $\gamma_{ij}(t,\cdot)$ is the unique constant-speed
length-minimizing geodesic from $x_j(t)$ to $x_i(t)$.
Moreover, the map $(t,s)\mapsto\gamma_{ij}(t,s)$ is $C^2$
in $t$ and smooth in $s$.
	We set
	\[
	J_{ij}(t,s):=\partial_t\gamma_{ij}(t,s),
	\quad
	T_{ij}(t,s):=\partial_s\gamma_{ij}(t,s),
	\]
	and
	\[
	W_{ij}(t,s):=
	P_{\gamma_{ij}(t,s)x_j(t)}v_j(t),
	\quad 0\le s\le1.
	\]
Then, for each $0\leq t<T$, one has
	\begin{align*}
		\nabla_{v_i(t)}\Big(P_{x_i(t)x_j(t)}v_j(t)\Big)
		&=
		P_{x_i(t)x_j(t)}\nabla_{v_j(t)}v_j(t)
		+
		\mathcal R_{ij}(t),
	\end{align*}
	where
	\begin{align*}
		\mathcal R_{ij}(t)
		:=
		\int_0^1
		P_{x_i(t)\gamma_{ij}(t,s)}
		\Big(
		R\big(T_{ij}(t,s),J_{ij}(t,s)\big)W_{ij}(t,s)
		\Big)\,ds.
	\end{align*}
	Consequently, one has
	\begin{align*}
		&\frac{d}{dt}
		\|P_{x_i(t)x_j(t)}v_j(t)-v_i(t)\|_{x_i(t)}^2 \nonumber\\
		&\quad =
		2g_{x_i(t)}
		\Big(
		P_{x_i(t)x_j(t)}\nabla_{v_j(t)}v_j(t)
		-
		\nabla_{v_i(t)}v_i(t)
		+
		\mathcal R_{ij}(t),
		\,
		P_{x_i(t)x_j(t)}v_j(t)-v_i(t)
		\Big).
	\end{align*}
\end{lemma}

\begin{proof}
	For simplicity, we omit the subscripts $ij$ and write
	\[
	\gamma(t,s)=\gamma_{ij}(t,s),\quad
	J(t,s)=J_{ij}(t,s),\quad
	T(t,s)=T_{ij}(t,s),\quad
	W(t,s)=W_{ij}(t,s).
	\]
	We also write $\nabla_t$ and $\nabla_s$ for the covariant derivatives along the $t$- and $s$-directions of the two-parameter map $\gamma(t,s)$, respectively. By the definition of $W$, we have
	\[
	\nabla_s W(t,s)=0,\quad W(t,0)=v_j(t),
	\quad W(t,1)=P_{x_i(t)x_j(t)}v_j(t).
	\]
	Using the curvature identity (see Definition \ref{D2.1})
	\[
	\nabla_s\nabla_t W-\nabla_t\nabla_s W
	=
	R(T,J)W,
	\]
	and the fact that $\nabla_sW=0$, we obtain
	\begin{equation}\label{new4}
		\nabla_s\nabla_t W
		=
		R(T,J)W,
	\end{equation}
where the Lie bracket term in the definition of the Riemannian curvature tensor disappears because the two parameters $s$ and $t$ commute, i.e., $[\partial_s,\partial_t]=0$.
	Now, parallel-transporting \eqref{new4} from $\gamma(t,s)$ to $x_i(t)=\gamma(t,1)$ and integrating over $s\in[0,1]$, we get
	\begin{align*}
		\nabla_t W(t,1)
		-
		P_{x_i(t)x_j(t)}\nabla_t W(t,0)
		&=
		\int_0^1
		P_{x_i(t)\gamma(t,s)}
		\big(
		\nabla_s\nabla_t W(t,s)
		\big)\,ds\\
		&=
		\int_0^1
		P_{x_i(t)\gamma(t,s)}
		\Big(
		R(T(t,s),J(t,s))W(t,s)
		\Big)\,ds.
	\end{align*}
	Since $W(t,0)=v_j(t)$ is a vector field along $x_j(t)$, we have
	\[
	\nabla_t W(t,0)=\nabla_{v_j(t)}v_j(t),
	\]
	and moreover, one has
	\[
	\nabla_t W(t,1)
	=
	\nabla_{v_i(t)}\Big(P_{x_i(t)x_j(t)}v_j(t)\Big).
	\]
	Therefore, we find
	\begin{align}\label{new5}
	\nabla_{v_i(t)}\Big(P_{x_i(t)x_j(t)}v_j(t)\Big)
	=
	P_{x_i(t)x_j(t)}\nabla_{v_j(t)}v_j(t)
	+
	\mathcal R_{ij}(t),
	\end{align}
	where
	\[
	\mathcal R_{ij}(t)
	=
	\int_0^1
	P_{x_i(t)\gamma_{ij}(t,s)}
	\Big(
	R\big(T_{ij}(t,s),J_{ij}(t,s)\big)W_{ij}(t,s)
	\Big)\,ds.
	\]
	This proves the desired first assertion.
	In what follows, we set
	\[
	Z_{ij}(t)
	:=
	P_{x_i(t)x_j(t)}v_j(t)-v_i(t)
	\in T_{x_i(t)}\mathcal M.
	\]
	By the metric compatibility of the Levi--Civita connection,
	\begin{align}\label{Shim}
	\frac{d}{dt}\|Z_{ij}(t)\|_{x_i(t)}^2
	=
	2g_{x_i(t)}
	\left(
	\nabla_{v_i(t)}Z_{ij}(t),Z_{ij}(t)
	\right).
	\end{align}
	Using \eqref{new5}, one can compute
	\begin{align}\label{new6}
	\nabla_{v_i(t)}Z_{ij}(t)
	=
	P_{x_i(t)x_j(t)}\nabla_{v_j(t)}v_j(t)
	-
	\nabla_{v_i(t)}v_i(t)
	+
	\mathcal R_{ij}(t).
	\end{align}
	Substituting \eqref{new6} into \eqref{Shim} yields
	\begin{align*}
		&\frac{d}{dt}
		\|P_{x_i(t)x_j(t)}v_j(t)-v_i(t)\|_{x_i(t)}^2\\
		&\quad =
		2g_{x_i(t)}
		\Big(
		P_{x_i(t)x_j(t)}\nabla_{v_j(t)}v_j(t)
		-
		\nabla_{v_i(t)}v_i(t)
		+
		\mathcal R_{ij}(t),
		\,
		P_{x_i(t)x_j(t)}v_j(t)-v_i(t)
		\Big).
	\end{align*}
	This completes the proof.
\end{proof}

{
	We first establish a uniform endpoint estimate for Jacobi fields along
	short length-minimizing geodesics.
	
	\begin{lemma}\label{L3.3}
		\emph{(Uniform endpoint estimate for Jacobi fields)}
		Assume that the Riemannian curvature tensor of $\mathcal M$ is
		uniformly bounded, i.e.,
		\[
		\|R_x\|_{\mathrm{op}}\leq C_R,
		\quad x\in\mathcal M.
		\]
		Let
		\[
		0<U<\operatorname{inj}(\mathcal M).
		\]
		Then, there exists a constant
		\[
		C_J=C_J(\mathcal M,U)>0
		\]
		such that, for every constant-speed length-minimizing geodesic
		\[
		\gamma:[0,1]\to\mathcal M
		\]
		satisfying
		\[
		L(\gamma)\leq U
		\]
		and every Jacobi field $J$ along $\gamma$, one has
		\[
		\sup_{0\leq s\leq1}
		\|J(s)\|_{\gamma(s)}
		\leq
		C_J
		\left(
		\|J(0)\|_{\gamma(0)}
		+
		\|J(1)\|_{\gamma(1)}
		\right).
		\]
	\end{lemma}
	
	\begin{proof}
		Choose $\Lambda>1$ such that
		\[
		\Lambda U<\operatorname{inj}(\mathcal M).
		\]
		Since $\mathcal M$ is complete, every geodesic
		$\gamma:[0,1]\to\mathcal M$ with $L(\gamma)\leq U$ can be extended
		with the same constant speed to a geodesic
		\[
		\gamma:[0,\Lambda]\to\mathcal M.
		\]
		The extended length satisfies
		\[
		L\bigl(\gamma|_{[0,\Lambda]}\bigr)
		=
		\Lambda L(\gamma)
		\leq
		\Lambda U
		<
		\operatorname{inj}(\mathcal M),
		\]
		and hence the extended geodesic is still length-minimizing.
		Suppose, to the contrary, that the asserted estimate fails. Then, there
		exist constant-speed length-minimizing geodesics
		\[
		\gamma_n:[0,\Lambda]\to\mathcal M
		\]
		with
		\[
		L\bigl(\gamma_n|_{[0,1]}\bigr)\leq U,
		\]
		and Jacobi fields $J_n$ along $\gamma_n|_{[0,1]}$ such that
		\[
		\frac{
			\|J_n(0)\|_{\gamma_n(0)}
			+
			\|J_n(1)\|_{\gamma_n(1)}
		}{
			\displaystyle
			\sup_{0\leq s\leq1}
			\|J_n(s)\|_{\gamma_n(s)}
		}
		\longrightarrow0.
		\]
		After normalization, we may assume that
		\begin{equation}\label{eq:normalized-jacobi}
			\sup_{0\leq s\leq1}
			\|J_n(s)\|_{\gamma_n(s)}
			=
			1,
			\quad
			\|J_n(0)\|_{\gamma_n(0)}
			+
			\|J_n(1)\|_{\gamma_n(1)}
			\longrightarrow0.
		\end{equation}
		Choose a parallel orthonormal frame
		\[
		\{E_{n,1}(s),\ldots,E_{n,d}(s)\}
		\]
		along $\gamma_n$. On $[0,1]$, write
		\[
		J_n(s)
		=
		\sum_{\alpha=1}^d
		y_n^\alpha(s)E_{n,\alpha}(s),
		\]
		and define
		\[
		\mathbf y_n(s)
		:=
		\bigl(
		y_n^1(s),\ldots,y_n^d(s)
		\bigr)^{\mathsf T}
		\in\mathbb R^d.
		\]
		Then, $\mathbf y_n$ satisfies
		\begin{equation}\label{eq:coordinate-jacobi}
			\mathbf y_n''(s)
			+
			\mathcal K_n(s)\mathbf y_n(s)
			=
			0,
			\quad 0\leq s\leq1,
		\end{equation}
		where the symmetric matrix $\mathcal K_n(s)$ is defined by
		\[
		\bigl(\mathcal K_n(s)\bigr)_{\alpha\beta}
		:=
		g_{\gamma_n(s)}
		\left(
		R\bigl(
		E_{n,\beta}(s),\dot\gamma_n(s)
		\bigr)
		\dot\gamma_n(s),
		E_{n,\alpha}(s)
		\right).
		\]
		Since $\gamma_n$ is parametrized with constant speed,
		\[
		\|\dot\gamma_n(s)\|_{\gamma_n(s)}
		=
		L\bigl(\gamma_n|_{[0,1]}\bigr)
		\leq U,
		\quad 0\leq s\leq\Lambda.
		\]
		Therefore, the curvature bound gives
		\begin{equation}\label{eq:curvature-matrix-bound}
			\|\mathcal K_n(s)\|_{\mathrm{op}}
			\leq
			C_RU^2,
			\quad 0\leq s\leq\Lambda.
		\end{equation}
		Since the frame is orthonormal, \eqref{eq:normalized-jacobi} implies
		\[
		\|\mathbf y_n\|_{L^\infty(0,1)}=1.
		\]
		Moreover, by \eqref{eq:coordinate-jacobi} and
		\eqref{eq:curvature-matrix-bound},
		\[
		\|\mathbf y_n''\|_{L^\infty(0,1)}
		\leq
		C_RU^2.
		\]
		For every $s\in[0,1]$, we have
		\[
		\mathbf y_n(1)-\mathbf y_n(0)
		=
		\int_0^1\mathbf y_n'(r)\,dr,
		\]
		and hence
		\[
		\mathbf y_n'(s)
		=
		\mathbf y_n(1)-\mathbf y_n(0)
		-
		\int_0^1
		\bigl(
		\mathbf y_n'(r)-\mathbf y_n'(s)
		\bigr)\,dr.
		\]
		Consequently,
		\begin{align*}
			|\mathbf y_n'(s)|
			&\leq
			|\mathbf y_n(1)|
			+
			|\mathbf y_n(0)|
			+
			\int_0^1
			|\mathbf y_n'(r)-\mathbf y_n'(s)|\,dr
			\\
			&\leq
			2\|\mathbf y_n\|_{L^\infty(0,1)}
			+
			\|\mathbf y_n''\|_{L^\infty(0,1)}
			\int_0^1|r-s|\,dr
			\\
			&\leq
			2+\frac12C_RU^2.
		\end{align*}
		Thus, $\mathbf y_n$, $\mathbf y_n'$, and $\mathbf y_n''$ are
		uniformly bounded on $[0,1]$.
		By the Arzel\`a--Ascoli theorem, after passing to a subsequence,
		\[
		\mathbf y_n
		\rightarrow
		\mathbf y
		\quad\text{in }C^1([0,1];\mathbb R^d).
		\]
		After taking a further subsequence if necessary, the
		Banach--Alaoglu theorem also gives
		\[
		\mathcal K_n
		\rightharpoonup^\ast
		\mathcal K
		\quad\text{in}\quad
		L^\infty
		\bigl(
		(0,\Lambda);\mathbb R^{d\times d}
		\bigr)
		\]
		for some symmetric matrix-valued function $\mathcal K$. Passing to
		the limit in \eqref{eq:coordinate-jacobi}, we obtain
		\[
		\mathbf y''+\mathcal K\mathbf y=0
		\quad\text{on }(0,1)
		\]
		in the distributional sense. Moreover,
		\eqref{eq:normalized-jacobi} gives
		\begin{equation}\label{eq:limiting-boundary-data}
			\mathbf y(0)=\mathbf y(1)=0,
			\quad
			\sup_{0\leq s\leq1}|\mathbf y(s)|=1.
		\end{equation}
		Since $\mathcal K\in L^\infty$ and $\mathbf y\in C^1$, the limiting
		equation implies
		\[
		\mathbf y\in W^{2,\infty}((0,1);\mathbb R^d).
		\]
	In particular, the integrations by parts used below are justified.
	Now, let
	\[
	\mathbf z:[0,\Lambda]\to\mathbb R^d
	\]
	be a continuous piecewise $C^1$ function satisfying
	\[
	\mathbf z(0)=\mathbf z(\Lambda)=0.
	\]
	Using the parallel frame along $\gamma_n$, define the vector field
	\[
	X_n^{\mathbf z}(s)
	:=
	\sum_{\alpha=1}^d
	z^\alpha(s)E_{n,\alpha}(s).
	\]
	Recall from \cite{Ju,P} that, for a vector field $X$ along $\gamma_n$ vanishing at the
	endpoints, the index form associated with $\gamma_n$ is given by
	\[
	I_{\gamma_n}(X,X)
	:=
	\int_0^\Lambda
	\left[
	\|\nabla_sX(s)\|^2
	-
	g\big(R(X(s),\dot\gamma_n(s))\dot\gamma_n(s),X(s)\big)
	\right]ds.
	\]
	Since $\gamma_n$ is a constant-speed length-minimizing geodesic on
	$[0,\Lambda]$, the second variation formula for the geodesic energy
	implies that
	\[
	I_{\gamma_n}(X,X)\ge0
	\]
	for every continuous piecewise $C^1$ vector field $X$ along $\gamma_n$
	vanishing at the endpoints (see \cite{Ju,P} for details). Applying
	this to $X_n^{\mathbf z}$ and using that the frame
	$\{E_{n,\alpha}\}_{\alpha=1}^d$ is parallel along $\gamma_n$, we obtain
	\begin{equation}\label{eq:second-variation-nonnegative}
		\mathcal Q_n[\mathbf z]
		:=
		\int_0^\Lambda
		\left[
		|\mathbf z'(s)|^2
		-
		\left\langle
		\mathcal K_n(s)\mathbf z(s),
		\mathbf z(s)
		\right\rangle_{\mathbb R^d}
		\right]ds
		\geq0.
	\end{equation}
		Passing to the limit in
		\eqref{eq:second-variation-nonnegative}, we obtain
		\begin{equation}\label{eq:limiting-quadratic-form}
			\mathcal Q[\mathbf z]
			:=
			\int_0^\Lambda
			\left[
			|\mathbf z'(s)|^2
			-
			\left\langle
			\mathcal K(s)\mathbf z(s),
			\mathbf z(s)
			\right\rangle_{\mathbb R^d}
			\right]ds
			\geq0.
		\end{equation}
		Extend $\mathbf y$ by zero from $[0,1]$ to $[0,\Lambda]$ and denote
		the resulting function by
		\[
		\widetilde{\mathbf y}(s)
		:=
		\begin{cases}
			\mathbf y(s), & 0\leq s\leq1,\\
			0, & 1\leq s\leq\Lambda.
		\end{cases}
		\]
		By \eqref{eq:limiting-boundary-data},
		$\widetilde{\mathbf y}$ is continuous and piecewise $C^1$, with
		\[
		\widetilde{\mathbf y}(0)
		=
		\widetilde{\mathbf y}(\Lambda)
		=
		0.
		\]
		Therefore, it is admissible in
		\eqref{eq:limiting-quadratic-form}. Using the limiting Jacobi
		equation on $[0,1]$ and integration by parts, we obtain
		\begin{align}
			\mathcal Q[\widetilde{\mathbf y}]
			&=
			\int_0^1
			\left[
			|\mathbf y'(s)|^2
			-
			\left\langle
			\mathcal K(s)\mathbf y(s),
			\mathbf y(s)
			\right\rangle_{\mathbb R^d}
			\right]ds
			\nonumber\\
			&=
			\left[
			\left\langle
			\mathbf y'(s),\mathbf y(s)
			\right\rangle_{\mathbb R^d}
			\right]_{s=0}^{s=1}
			=
			0.
			\label{eq:zero-quadratic-form}
		\end{align}
		Define the associated symmetric bilinear form by
		\[
		\mathcal B(\mathbf u,\mathbf z)
		:=
		\int_0^\Lambda
		\left[
		\left\langle
		\mathbf u'(s),\mathbf z'(s)
		\right\rangle_{\mathbb R^d}
		-
		\left\langle
		\mathcal K(s)\mathbf u(s),
		\mathbf z(s)
		\right\rangle_{\mathbb R^d}
		\right]ds.
		\]
		Then, one has
		\[
		\mathcal Q[\mathbf z]
		=
		\mathcal B(\mathbf z,\mathbf z).
		\]
		For every admissible $\mathbf z$ and every
		$\varepsilon\in\mathbb R$, relations
		\eqref{eq:limiting-quadratic-form} and
		\eqref{eq:zero-quadratic-form} give
		\[
		\begin{aligned}
			0
			&\leq
			\mathcal Q
			\bigl[
			\widetilde{\mathbf y}
			+
			\varepsilon\mathbf z
			\bigr]
			\\
			&=
			2\varepsilon
			\mathcal B
			\bigl(
			\widetilde{\mathbf y},
			\mathbf z
			\bigr)
			+
			\varepsilon^2
			\mathcal Q[\mathbf z].
		\end{aligned}
		\]
		Considering both positive and negative $\varepsilon$, we conclude
		that
		\begin{equation}\label{eq:bilinear-vanishing}
			\mathcal B
			\bigl(
			\widetilde{\mathbf y},
			\mathbf z
			\bigr)
			=
			0
		\end{equation}
		for every admissible $\mathbf z$.
		On the other hand, the limiting Jacobi equation and integration by
		parts yield
		\begin{align*}
			\mathcal B
			\bigl(
			\widetilde{\mathbf y},
			\mathbf z
			\bigr)
			&=
			\int_0^1
			\left[
			\left\langle
			\mathbf y'(s),\mathbf z'(s)
			\right\rangle_{\mathbb R^d}
			-
			\left\langle
			\mathcal K(s)\mathbf y(s),
			\mathbf z(s)
			\right\rangle_{\mathbb R^d}
			\right]ds
			\\
			&=
			\left[
			\left\langle
			\mathbf y'(s),\mathbf z(s)
			\right\rangle_{\mathbb R^d}
			\right]_{s=0}^{s=1}
			\\
			&=
			\left\langle
			\mathbf y'(1),
			\mathbf z(1)
			\right\rangle_{\mathbb R^d},
		\end{align*}
		where we used $\mathbf z(0)=0$. Since admissible test functions can
		be chosen with arbitrary value at $s=1$,
		\eqref{eq:bilinear-vanishing} implies
		\[
		\mathbf y'(1)=0.
		\]
		Together with $\mathbf y(1)=0$, uniqueness for the linear
		second-order equation with bounded coefficients gives
		\[
		\mathbf y\equiv0
		\quad\text{on }[0,1],
		\]
		which contradicts
		\[
		\sup_{0\leq s\leq1}|\mathbf y(s)|=1.
		\]
		This completes the proof.
	\end{proof}
}

	Using the variation formula in Lemma~\ref{L3.2} and the uniform
	Jacobi-field estimate in Lemma~\ref{L3.3}, we next obtain a uniform
	derivative bound for the squared transported velocity discrepancies.
	
	\begin{proposition}\label{P3.3}
		{
			\emph{(Uniform derivative bound for squared transported velocity
				discrepancies)}
		}
	{
		Assume $(\mathcal F_1)$, $(\mathcal F_2)$, and $(\mathcal F_4)$,
		and let $\{(x_i,v_i)\}_{i=1}^N$ be the unique global solution to
		\eqref{Main} given by Proposition~\ref{P3.1}.
	} Then there
		exists a constant $C>0$, independent of $t$, such that
		\[
		\max_{i,j\in[N]}
		\left|
		\frac{d}{dt}
		\|P_{x_i(t)x_j(t)}v_j(t)-v_i(t)\|_{x_i(t)}^2
		\right|
		\leq C,
		\quad t>0.
		\]
	\end{proposition}
	
	\begin{proof}
		It is enough to consider the case $i\neq j$, since the case $i=j$ is
		trivial. Throughout the proof, we use the notation
		\[
		d_{ij}(t):=d(x_i(t),x_j(t)),
		\quad
		P_{ij}(t):=P_{x_i(t)x_j(t)}.
		\]
		By Proposition~\ref{P3.1}, there exists a constant
		\[
		V:=\sqrt{2\mathcal E(0)}
		\]
		such that
		\begin{align}\label{estimate}
			\max_{i\in[N]}\|v_i(t)\|_{x_i(t)}
			&\leq V,
			\nonumber\\
			\max_{i,j\in[N]}d_{ij}(t)
			&\leq U<\operatorname{inj}(\mathcal M),
			\quad t\geq0.
		\end{align}
		In particular, all logarithm maps and parallel transports appearing
		below are well-defined.
		Since $\phi_f$ and $\phi_b$ are locally Lipschitz continuous, they
		are bounded on compact intervals. We set
		\[
		\overline{\phi_f}
		:=
		\sup_{0\leq r\leq U}\phi_f(r)
		<\infty,
		\quad
		\overline{\phi_b}
		:=
		\sup_{0\leq r\leq U^2}\phi_b(r)
		<\infty.
		\]
		Using \eqref{Main}, the velocity bound in \eqref{estimate}, and
		\eqref{new2}, we obtain
		\begin{align*}
			\|\nabla_{v_i(t)}v_i(t)\|_{x_i(t)}
			&\leq
			\frac{\kappa_f}{N}
			\sum_{k=1}^N
			\phi_f(d(x_i(t),x_k(t)))
			\|P_{x_i(t)x_k(t)}v_k(t)-v_i(t)\|_{x_i(t)}
			\\
			&\quad+
			\frac{\kappa_b}{N}
			\sum_{k=1}^N
			\phi_b(d(x_i(t),x_k(t))^2)
			\|\log_{x_i(t)}x_k(t)\|_{x_i(t)}
			\\
			&\leq
			2\kappa_f\overline{\phi_f}V
			+
			\kappa_b\overline{\phi_b}U.
		\end{align*}
		Hence, there exists a constant $A>0$, independent of $i$ and $t$,
		such that
		\begin{equation}\label{C0}
			\max_{i\in[N]}
			\|\nabla_{v_i(t)}v_i(t)\|_{x_i(t)}
			\leq A,
			\quad t>0.
		\end{equation}
		We next estimate the curvature remainder in Lemma~\ref{L3.2}.
		For each $i,j\in[N]$ with $i\neq j$, let
		\[
		\gamma_{ij}(t,\cdot):[0,1]\to\mathcal M
		\]
		be the unique constant-speed length-minimizing geodesic from
		$x_j(t)$ to $x_i(t)$. We set
		\[
		J_{ij}(t,s):=\partial_t\gamma_{ij}(t,s),
		\quad
		T_{ij}(t,s):=\partial_s\gamma_{ij}(t,s),
		\]
		and
		\[
		W_{ij}(t,s)
		:=
		P_{\gamma_{ij}(t,s)x_j(t)}v_j(t),
		\quad 0\leq s\leq1.
		\]
		For notational simplicity, we omit the base points from the norms of
		vector fields along $\gamma_{ij}(t,\cdot)$.
		The vector field $J_{ij}(t,\cdot)$ is a Jacobi field along
		$\gamma_{ij}(t,\cdot)$ with endpoint values
		\begin{align}\label{C1}
			J_{ij}(t,0)=v_j(t),
			\quad
			J_{ij}(t,1)=v_i(t).
		\end{align}
		Moreover, we have
		\[
		L\bigl(\gamma_{ij}(t,\cdot)\bigr)
		=
		d_{ij}(t)
		\leq U.
		\]
		Therefore, Lemma~\ref{L3.3}, \eqref{C1}, and the velocity bound imply
		\[
		\begin{aligned}
			\sup_{0\leq s\leq1}\|J_{ij}(t,s)\|
			&\leq
			C_J
			\left(
			\|J_{ij}(t,0)\|
			+
			\|J_{ij}(t,1)\|
			\right)
			\\
			&=
			C_J
			\left(
			\|v_j(t)\|_{x_j(t)}
			+
			\|v_i(t)\|_{x_i(t)}
			\right)
			\\
			&\leq
			2C_JV.
		\end{aligned}
		\]
		Since $\gamma_{ij}(t,\cdot)$ is parametrized with constant speed and
		parallel transport is an isometry, we also have
		\[
		\|T_{ij}(t,s)\|
		=
		d_{ij}(t)
		\leq U
		\]
		and
		\[
		\|W_{ij}(t,s)\|
		=
		\|v_j(t)\|_{x_j(t)}
		\leq V,
		\quad 0\leq s\leq1.
		\]
		Using assumption $(\mathcal F_4)$, we obtain
		\begin{align*}
			\|\mathcal R_{ij}(t)\|_{x_i(t)}
			&\leq
			\int_0^1
			\left\|
			R\bigl(
			T_{ij}(t,s),J_{ij}(t,s)
			\bigr)
			W_{ij}(t,s)
			\right\|\,ds
			\\
			&\leq
			\int_0^1
			C_R
			\|T_{ij}(t,s)\|
			\|J_{ij}(t,s)\|
			\|W_{ij}(t,s)\|\,ds
			\\
			&\leq
			2C_RC_JUV^2.
		\end{align*}
		Hence, there exists a constant $B>0$, independent of $i$, $j$, and
		$t$, such that
		\begin{equation}\label{eq:curvature-remainder-bound}
		\|\mathcal R_{ij}(t)\|_{x_i(t)}
		\leq B,
		\quad
		i,j\in[N],\quad i\neq j,\quad t\geq0.
		\end{equation}
		Finally, Lemma~\ref{L3.2} yields
		\begin{align*}
			&\frac{d}{dt}
			\|P_{ij}(t)v_j(t)-v_i(t)\|_{x_i(t)}^2
			\\
			&\quad=
			2g_{x_i(t)}
			\Bigl(
			P_{ij}(t)\nabla_{v_j(t)}v_j(t)
			-
			\nabla_{v_i(t)}v_i(t)
			+
			\mathcal R_{ij}(t),
			P_{ij}(t)v_j(t)-v_i(t)
			\Bigr).
		\end{align*}
		Using \eqref{new1}, \eqref{C0},
		\eqref{eq:curvature-remainder-bound}, the triangle inequality, and
		the velocity bound, we obtain
		\begin{align*}
			&
			\left|
			\frac{d}{dt}
			\|P_{ij}(t)v_j(t)-v_i(t)\|_{x_i(t)}^2
			\right|
			\\
			&\leq
			2
			\Bigl(
			\|\nabla_{v_j(t)}v_j(t)\|_{x_j(t)}
			+
			\|\nabla_{v_i(t)}v_i(t)\|_{x_i(t)}
			+
			\|\mathcal R_{ij}(t)\|_{x_i(t)}
			\Bigr)
			\|P_{ij}(t)v_j(t)-v_i(t)\|_{x_i(t)}
			\\
			&\leq
			2(2A+B)
			\left(
			\|v_j(t)\|_{x_j(t)}
			+
			\|v_i(t)\|_{x_i(t)}
			\right)
			\\
			&\leq
			4V(2A+B).
		\end{align*}
		Therefore, setting
		\[
		C:=4V(2A+B),
		\]
		we conclude that
		\[
		\max_{i,j\in[N]}
		\left|
		\frac{d}{dt}
		\|P_{x_i(t)x_j(t)}v_j(t)-v_i(t)\|_{x_i(t)}^2
		\right|
		\leq C,
		\quad t>0.
		\]
		This completes the proof.
	\end{proof}

Finally, we are ready to prove the desired asymptotic flocking of \eqref{Main} under $(\mathcal F_1)-(\mathcal F_4)$.

\begin{theorem} \label{T3.1}~\emph{(Asymptotic flocking)}
{
	Assume $(\mathcal F_1)-(\mathcal F_4)$, and let
	$\{(x_i,v_i)\}_{i=1}^N$ be the unique global solution to
	\eqref{Main} given by Proposition~\ref{P3.1}.
} Then, the solution exhibits asymptotic flocking in the sense of Definition \ref{D1.1}. More precisely,
	\begin{equation*}
		\sup_{t\ge0}\max_{i,j\in[N]}d(x_i(t),x_j(t))
		\le
		U
		<
		\operatorname{inj}(\mathcal M),
	\end{equation*}
	and
	\begin{equation*}
		\lim_{t\to\infty}
		\max_{i,j\in[N]}
		\|P_{x_i(t)x_j(t)}v_j(t)-v_i(t)\|_{x_i(t)}
		=0.
	\end{equation*}
\end{theorem}

\begin{proof}
	By Proposition \ref{P3.3}, there exists a positive constant $C>0$ such that, for all $i,j\in[N]$ and $t>0$,
	\begin{equation*}
		\left|
		\frac{d}{dt}
		\|P_{x_i(t)x_j(t)}v_j(t)-v_i(t)\|_{x_i(t)}^2
		\right|
		\le C.
	\end{equation*}
	Hence, for each $i,j\in[N]$, the functional
	\[
	t\mapsto
	\|P_{x_i(t)x_j(t)}v_j(t)-v_i(t)\|_{x_i(t)}^2
	\]
	is uniformly continuous on $[0,\infty)$.
		On the other hand, Proposition \ref{P3.2} yields the time integrability of the transported velocity discrepancies:
	\begin{equation*}
		\int_0^\infty
		\|P_{x_i(t)x_j(t)}v_j(t)-v_i(t)\|_{x_i(t)}^2\,dt
		<
		\infty,
		\quad i,j\in[N].
	\end{equation*}
	Therefore, employing Barbalat's lemma (see Lemma \ref{barbalat}), we obtain
	\[
	\lim_{t\to\infty}
	\|P_{x_i(t)x_j(t)}v_j(t)-v_i(t)\|_{x_i(t)}^2
	=0,
	\quad i,j\in[N],
	\]
	which implies that
	\[
	\lim_{t\to\infty}
	\max_{i,j\in[N]}
	\|P_{x_i(t)x_j(t)}v_j(t)-v_i(t)\|_{x_i(t)}
	=0.
	\]
	Subsequently, we recall that the second assertion of Proposition \ref{P3.1} gives the uniform boundedness of the inter-agent distances:
	\[
	\sup_{t\ge0}\max_{i,j\in[N]}d(x_i(t),x_j(t))
	\le
	U
	<
	\operatorname{inj}(\mathcal M).
	\]
	Combining this distance bound with the asymptotic velocity alignment above, we conclude that the solution exhibits asymptotic flocking in the sense of Definition \ref{D1.1}. This completes the proof.
\end{proof}

\section{Numerical simulations} \label{sec:4}
\setcounter{equation}{0}
In this section, we provide numerical simulations for the system \eqref{Main}. We first present an admissible example on the product manifold $\mathbb S^2\times\mathbb R$ satisfying the sufficient conditions $(\mathcal F_1)-(\mathcal F_4)$. This manifold has bounded Riemannian curvature and is not a constant-sectional-curvature space. Hence, it provides a useful example beyond the standard geometric settings previously considered, such as Euclidean spaces, spheres, hyperboloids, and infinite cylindrical domains. We then present numerical results to support the theoretical results obtained in Section \ref{sec:3}.

\subsection{An admissible example on $\mathbb S^2\times\mathbb R$} \label{sec:4.1}
We first consider the product manifold
\[
\mathcal M=\mathbb S^2\times\mathbb R
\]
equipped with the product metric, where
\[
\mathbb S^2:=\{p\in\mathbb R^3~|~\|p\|_{\mathbb R^3}=1\}
\]
denotes the unit sphere in $\mathbb R^3$, and $\|\cdot\|_{\mathbb R^3}$ is the standard Euclidean norm. This manifold is noncompact and has bounded Riemannian curvature, but it is not a constant-sectional-curvature space. Indeed, the sectional curvature of the plane tangent to the $\mathbb S^2$-factor is $1$, while the sectional curvature of every mixed plane spanned by one direction tangent to $\mathbb S^2$ and one direction tangent to $\mathbb R$ is zero. Thus, this example is beyond the standard constant-curvature settings. At the same time, the product structure allows us to compute the geodesic distance, logarithm map, and parallel transport componentwise. The geometric formulas on $\mathbb S^2$ used below are standard and can be found, for example, in \cite{A-H-S,A-B-H-Y}.
Let
\[
N=10,\quad \kappa_f=\kappa_b=1,
\]
and choose the kernels
\[
\phi_f(r)=1,\quad \phi_b(r)=1,\quad r\ge0.
\]
Then, $(\mathcal F_1)$ is satisfied. Moreover, since $\phi_f\equiv1$, we have
\[
\underline{\phi_f}
=
\inf_{0\le r\le U}\phi_f(r)
=
1>0,
\]
so that $(\mathcal F_3)$ holds.
For $x=(p,z),y=(q,w)\in\mathbb S^2\times\mathbb R$, where $p,q\in\mathbb S^2$ and $z,w\in\mathbb R$, the distance is given by
\[
d_{\mathcal M}(x,y)^2
=
d_{\mathbb S^2}(p,q)^2+|z-w|_{\mathbb R}^2,
\quad
d_{\mathbb S^2}(p,q)=\arccos(p\cdot q),
\]
where $p\cdot q$ is the standard Euclidean inner product in $\mathbb R^3$, and $|\cdot|_{\mathbb R}$ denotes the standard Euclidean distance on the $\mathbb R$-factor.
If $p\neq -q$, then the logarithm map is
\[
\log_x y
=
\left(
\log_p^{\mathbb S^2}q,\,
w-z
\right),
\]
where
\[
\log_p^{\mathbb S^2}q
=
\frac{\theta}{\sin\theta}(q-\cos\theta\,p),
\quad
\theta:=\arccos(p\cdot q).
\]
For $v=(u,\xi)\in T_y\mathcal M=T_q\mathbb S^2\times\mathbb R$, the parallel transport from $y=(q,w)$ to $x=(p,z)$ is
\[
P_{xy}v
=
\left(
P_{pq}^{\mathbb S^2}u,\,
\xi
\right),
\]
where
\[
P_{pq}^{\mathbb S^2}u
=
u-\frac{u\cdot p}{1+p\cdot q}(p+q),
\quad p\neq -q.
\]
Additionally, let
\[
x(t)=(p(t),z(t))\in\mathbb S^2\times\mathbb R
\]
be a smooth curve, and let
\[
V(t)=(Y(t),\eta(t))
\]
be a vector field along $x(t)$. Then the product connection is given by
\[
\nabla_{\dot x(t)}^{\mathbb S^2\times\mathbb R}V(t)
=
\left(
\nabla_{\dot p(t)}^{\mathbb S^2}Y(t),
\dot\eta(t)
\right).
\]
In particular, if $p(t)\in\mathbb S^2$ and $u(t)=\dot p(t)$, then in the ambient space $\mathbb R^3$,
\[
\nabla^{\mathbb S^2}_{u(t)}u(t)
=
\dot u(t)+\|u(t)\|_{\mathbb R^3}^2p(t).
\]
We now choose admissible initial data. Let
\[
\varepsilon=\delta=10^{-2},
\quad
\alpha_i:=\frac{2\pi(i-1)}{10},
\quad i\in [10],
\]
and then we set
\[
p_i^0
:=
\left(
\sqrt{1-\varepsilon^2},\,
\varepsilon\cos\alpha_i,\,
\varepsilon\sin\alpha_i
\right)
\in\mathbb S^2,
\quad
z_i^0:=\varepsilon\cos\alpha_i,
\]
and
\[
x_i^0:=(p_i^0,z_i^0)\in\mathbb S^2\times\mathbb R.
\]
For the initial velocities, we define
\[
u_i^0
:=
\delta
\left(
0,\,
-\sin\alpha_i,\,
\cos\alpha_i
\right)
\in T_{p_i^0}\mathbb S^2,
\quad
\xi_i^0:=\delta\sin\alpha_i,
\]
and
\[
v_i^0:=(u_i^0,\xi_i^0)\in T_{x_i^0}\mathcal M.
\]
Then, one has $p_i^0\cdot u_i^0=0$, and hence $u_i^0$ is tangent to $\mathbb S^2$ at $p_i^0$.
Since $\phi_b\equiv1$, the initial energy satisfies
\[
\mathcal E(0)
=
\frac12\sum_{i=1}^{10}\|v_i^0\|_{x_i^0}^2
+
\frac{1}{40}
\sum_{i,j=1}^{10}
d(x_i^0,x_j^0)^2.
\]
For the above initial data with $\varepsilon=\delta=10^{-2}$, we have
\[
\mathcal E(0)
=
\frac{3}{4}\cdot10^{-3}
+
\frac14
\sum_{k=0}^{9}
\left[
\arccos\left(
1-10^{-4}\left(1-\cos\frac{2\pi k}{10}\right)
\right)^2
+
10^{-4}\left(1-\cos\frac{2\pi k}{10}\right)
\right].
\]
Numerically, this gives
\[
\mathcal E(0)
\approx
1.5000125006\times 10^{-3}.
\]
Moreover, for $\phi_b\equiv1$ and $N=10$, the bonding-energy radius is
\[
U
=
\sup\left\{
r\ge0~\bigg|~
\frac{1}{20}r^2\le \mathcal E(0)
\right\}
=
\sqrt{20\mathcal E(0)}.
\]
Therefore, we find
\[
U
\approx
0.1732058025
<
\pi
=
\operatorname{inj}(\mathbb S^2\times\mathbb R).
\]
Hence, $(\mathcal F_2)$ is satisfied. Finally, since the $\mathbb S^2$-factor has constant sectional
curvature $1$ and the $\mathbb R$-factor is flat, the curvature tensor
of $\mathbb S^2\times\mathbb R$ is uniformly bounded. In fact, one may
take $C_R=1$. Hence, $(\mathcal F_4)$ holds. Consequently, the above choice gives an admissible example satisfying $(\mathcal F_1)-(\mathcal F_4)$.

\subsection{Numerical simulations} \label{sec:4.2}
In this subsection, we present numerical results for the admissible example introduced in Section \ref{sec:4.1}. Specifically, we illustrate the uniform boundedness of inter-agent distances and the decay of the transported velocity discrepancies. For the simulation, we use the explicit Euler method with time step $h=0.001$. The simulation is performed on the time interval
\[
0\leq t\leq20.
\] 
Since the direct Euler update in the ambient Euclidean coordinates does not necessarily preserve the constraint $p_i(t)\in\mathbb S^2$, we project the spherical component onto $\mathbb S^2$ after each time step. We also project the spherical velocity component onto the tangent space $T_{p_i(t)}\mathbb S^2$. The $\mathbb R$-component is updated by the standard Euler method.

\begin{figure}[h!]
	\includegraphics[width=0.86\linewidth]{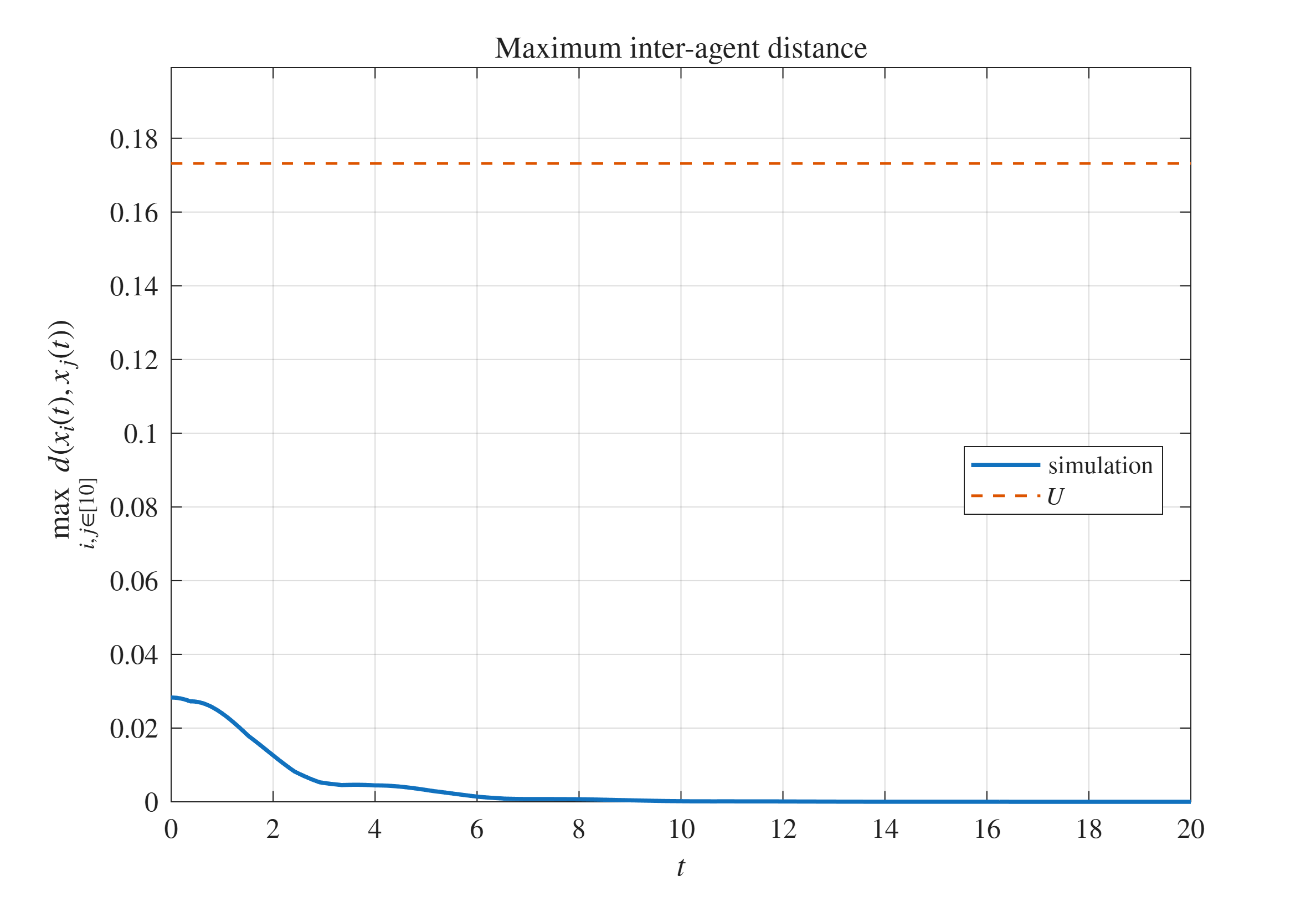}
	\caption{Time evolution of $\max_{i,j\in[10]}d(x_i,x_j)$}
	\label{first}
\end{figure}

Figure~\ref{first} shows the time evolution of the maximum inter-agent
distance
\[
\max_{i,j\in[10]}d(x_i(t),x_j(t)).
\]
As shown in the figure, the maximum inter-agent distance remains below the bonding-energy radius
\[
U\approx 0.1732058025,
\]
computed in Section~\ref{sec:4.1}. This numerically supports the second assertion of Proposition~\ref{P3.1}, namely the uniform distance estimate
\[
d(x_i(t),x_j(t))\le U,\quad i,j\in[10],\quad t\ge0.
\]
In this particular simulation, the maximum inter-agent distance also appears to decrease toward zero. However, this stronger behavior is not part of our theoretical result. It should be understood as a feature of the present numerical setting: the bonding kernel is constant, $\phi_b\equiv1$, so the bonding force is purely attractive, and the trajectories remain in a small localized region of $\mathbb S^2\times\mathbb R$ where the product geometry is close to the Euclidean one. Thus, the observed decrease of inter-agent distances should not be interpreted as a general consequence of our theoretical results.

\begin{figure}[h!]
	\includegraphics[width=0.86\linewidth]{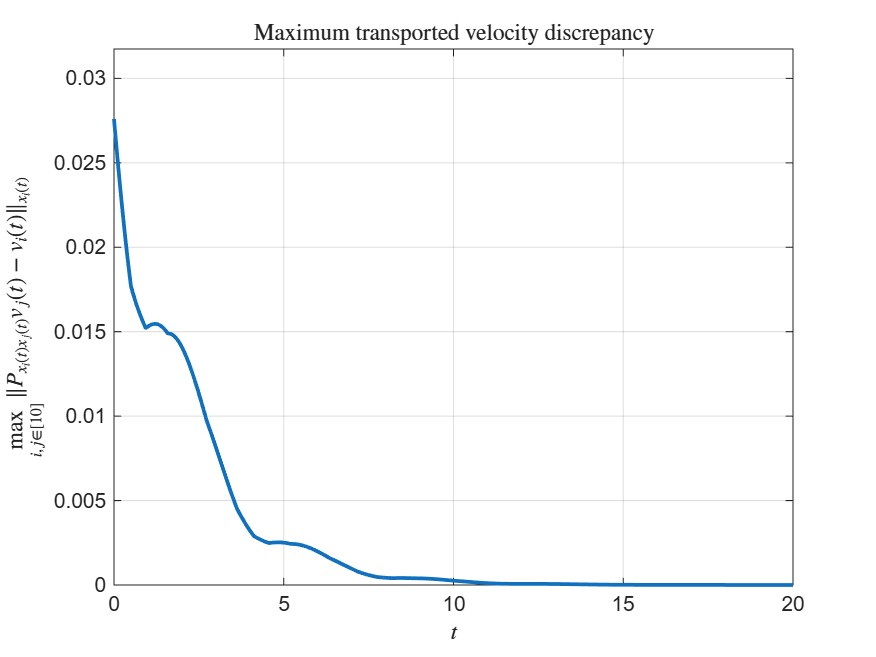}
	\caption{Time evolution of $\max_{i,j\in[10]}\|P_{x_ix_j}v_j-v_i\|_{x_i}$}
	\label{second}
\end{figure}

Next, Figure~\ref{second} shows the time evolution of the maximum transported velocity discrepancy
\[
\max_{i,j\in[10]}
\|P_{x_i(t)x_j(t)}v_j(t)-v_i(t)\|_{x_i(t)}.
\]
We observe that this quantity decays to zero as time evolves. This numerical result supports the asymptotic velocity-alignment estimate in Theorem~\ref{T3.1}, namely
\[
\lim_{t\to\infty}
\max_{i,j\in[10]}
\|P_{x_i(t)x_j(t)}v_j(t)-v_i(t)\|_{x_i(t)}
=0.
\]
Thus, together with Figure~\ref{first}, the simulation illustrates the asymptotic flocking behavior predicted by our theoretical result. Moreover, since the example is posed on $\mathbb S^2\times\mathbb R$, which has bounded Riemannian curvature but is not a constant-sectional-curvature space, the numerical results also demonstrate the applicability of our framework beyond the standard constant-curvature settings.

\section{Conclusion}
\label{sec:5}

In this paper, we established a geometric closure mechanism for asymptotic flocking in a Riemannian Cucker--Smale system with bonding forces under bounded, not necessarily constant, curvature. The bonding interaction plays the distinct role of confinement: under an energy-dependent injectivity condition, it keeps all pairwise distances below the injectivity radius. Together with the kinetic-energy estimate, this yields global well-posedness and uniform velocity bounds, while the energy dissipation gives
\[
\int_0^\infty
\sum_{i,j=1}^N
\|P_{x_i(t)x_j(t)}v_j(t)-v_i(t)\|_{x_i(t)}^2\,dt
<
\infty.
\]
The alignment conclusion requires an additional geometric argument, because confinement does not control the time variation of parallel transport. By combining the variation of parallel transport along moving minimizing geodesics with a uniform endpoint estimate for the associated Jacobi fields, we obtained
\[
\sup_{t>0}
\max_{i,j\in[N]}
\left|
\frac{d}{dt}
\|P_{x_i(t)x_j(t)}v_j(t)-v_i(t)\|_{x_i(t)}^2
\right|
<
\infty.
\]
Barbalat's lemma then yields asymptotic velocity alignment and hence asymptotic flocking. The main contribution is therefore the control of the curvature terms generated by moving parallel transport using only an injectivity-radius confinement and a uniform curvature bound, rather than explicit formulas or constant-curvature cancellations. This argument may also be useful for other collective dynamics models on curved spaces in which states at moving base points must be compared geometrically.

\end{document}